\documentclass[sn-mathphys-num]{sn-jnl}% Math and Physical Sciences Numbered Reference Style 
\usepackage{graphicx}%
\usepackage{multirow}%
\usepackage{amsmath,amssymb,amsfonts}%
\usepackage{amsthm}%
\usepackage{mathrsfs}%
\usepackage[title]{appendix}%
\usepackage{xcolor}%
\usepackage{textcomp}%
\usepackage{manyfoot}%
\usepackage{booktabs}%
\usepackage{algorithm}%
\usepackage{algorithmicx}%
\usepackage{algpseudocode}%
\usepackage{listings}%
\newtheorem{theorem}{Theorem}%  meant for continuous numbers
\newtheorem{problem}{Open Problem}%
\newtheorem{observation}{Observation}%
\newtheorem{corollary}{Corollary}%
\newtheorem{lemma}{Lemma}%

\begin{document}

\title[Local (Outer) Multiset Dimensions of Graphs]{Local (Outer) Multiset Dimensions of Graphs}


\author*[1,4]{\fnm{Rinovia} \sur{Simanjuntak}}\email{rino@itb.ac.id}

\author[2]{\fnm{M. Ali} \sur{Hasan}}\email{malihasan1996@gmail.com}

\author[2,3]{\fnm{Muhung} \sur{Anggarawan}}\email{muhung.anggarawan@digitalbdg.ac.id}

\affil[1]{\orgdiv{Combinatorial Mathematics Research Group }, \orgname{Faculty of Mathematics and Natural Sciences, Institut Teknologi Bandung}, \orgaddress{\street{Jalan Ganesa 10}, \city{Bandung}, \postcode{40132}, \state{Jawa Barat}, \country{Indonesia}}}

\affil[2]{\orgdiv{Master's Program in Mathematics}, \orgname{Faculty of Mathematics and Natural Sciences, Institut Teknologi Bandung}, \orgaddress{\street{Jalan Ganesa 10}, \city{Bandung}, \postcode{40132}, \state{Jawa Barat}, \country{Indonesia}}}

\affil[3]{\orgname{Akademi Digital Bandung}, \orgaddress{\street{Jalan Ranca Mekar Blok Lio}, \city{Bandung}, \postcode{40292}, \state{Jawa Barat}, 
\country{Indonesia}}}

\affil[4]{\orgdiv{Center for Research Collaboration in Graph Theory and Combinatorics}, 
%\orgname{Organization}, \orgaddress{\street{Street}, \city{City}, \postcode{610101}, \state{State}, 
\country{Indonesia}}


\abstract{Let $G$ be a finite, connected, undirected, and simple graph and $W$ be a set of vertices in $G$. A representation multiset of a vertex $u$ in $V(G)$ with respect to $W$ is defined as the multiset of distances between $u$ and the vertices in $W$. If every two adjacent vertices in $V(G)$ have a distinct multiset representation, the set $W$ is called a local multiset resolving set of $G$. If $G$ has a local multiset resolving set, then this set with the smallest cardinality is called the local multiset basis, and its cardinality is the local multiset dimension of $G$. Otherwise, $G$ is said to have an infinite local multiset dimension.\\ 

On the other hand, if every two adjacent vertices in $V(G) \backslash W$ have a distinct representation multiset, the set $W$ is called a local outer multiset resolving set of $G$. Such a set with the smallest cardinality is called the local outer multiset basis, and its cardinality is the local outer multiset dimension of $G$. Unlike the local multiset dimension, every graph has a finite local outer multiset dimension.\\

This paper presents some basic properties of local (outer) multiset dimensions, including lower bounds for the dimensions and a necessary condition for a finite local multiset dimension. We also determine local (outer) multiset dimensions for some graphs of small diameter. }




\maketitle

\section{Introduction }
\label{Intro}

In recent decades, there has been increased interest in studying ways to identify the vertices of a given graph uniquely. One of the best-known vertex identification methods is the metric dimension introduced by Slater \cite{Sl75} and Harary and Melter \cite{HM76} in the 1970s.\\

Throughout this paper, let $G$ be a finite, connected, undirected, and simple graph, with the vertex set $V(G)$ and the edge set $E(G)$. %For two vertices $u$ and $v$, {\em the distance} between $u$ and $v$ in $G$, denoted by $d_G(u,v)$, is the shortest path that travels from $u$ to $v$. In case $u=v$, $d_G(u,v)=0$. 
Let $W=\{w_1, \ldots, w_k\}$ be an ordered set of vertices in $G$ and let $u$ be a vertex of $G$. A {\em representation of $u$ with respect to $W$}, is a vector $r(u|W)=(d(u,w_1), \ldots, d(u,w_k))$, where $d(u,w)=d_G(u,w)$ is the distance between the vertices $u$ and $w$ in $G$. If for every pair of distinct vertices $u,v$ in $V(G)$, $r(u|W) \neq r(v|W)$, then $W$ is called {\em a resolving set for $G$}. A resolving set with the smallest cardinality is called a {\em basis}, and its cardinality is the {\em dimension of $G$}, denoted by $dim(G)$. However, a more common problem in graph theory concerns distinguishing every two neighbors rather than distinguishing all the vertices of $G$. Thus, if for every pair of adjacent vertices $u,v$ in $V(G)$, $r(u|W) \neq r(v|W)$, then $W$ is called {\em a local resolving set for $G$}. A local resolving set with the smallest cardinality is called a {\em local basis}, and its cardinality is the {\em local dimension of $G$}, denoted by $ldim(G)$. It is obvious that every basis of a graph is also its local basis, and so $1 \leq ldim(G) \leq dim(G) \leq n-1$ \cite{OPZ10}.\\

If we utilize a set of vertices $W=\{w_1,\ldots,w_k\}$, instead of an ordered set, then the representation of a vertex will be a multiset instead of a vector. Let $u$ be a vertex of $G$. A {\em representation multiset of $u$ with respect to $W$}, denoted by $m(u|W)$, is a multiset of distances between $u$ and all vertices in $W$. If for every pair of distinct vertices $u,v$ in $V(G)$, $m(u|W)\neq m(v|W)$, then $W$ is called {\em a multiset resolving set for $G$}. If $G$ has a multiset resolving set, then this set with the smallest cardinality is a {\em multiset basis}, and its cardinality is the {\em multiset dimension of $G$}, denoted by $md(G)$; otherwise, $G$ is said to have an infinite multiset dimension \cite{SVM17}. Every multiset basis of a graph is its basis, so $dim(G) \leq md(G)$.  The following theorem provides two necessary conditions for a graph to admit an infinite multiset dimension.\\

\begin{theorem} \cite{SVM17} \label{infmd}
    \begin{enumerate}
        \item If diameter $G$ is at most two, then $G$ has an infinite multiset dimension.
        \item If $G$ contains three vertices with the same open neighbourhood, then $G$ has an infinite multiset dimension.
    \end{enumerate}
\end{theorem}

To avoid an infinite multiset dimension, Gil-Pons \textit{et. al.} \cite{GRTY19} introduced the concept of an outer multiset dimension. Here, only vertices outside a multiset resolving set need to have distinct representation multiset. Formally, if for every pair of distinct vertices $u,v$ in $V(G)\backslash W$, $m(u|W)\neq m(v|W)$, then $W$ is called an {\em outer multiset resolving set for $G$}. This set with the smallest cardinality is an {\em outer multiset basis}, and its cardinality is the {\em outer multiset dimension of $G$}, denoted by $dim_{ms}(G)$. It has been shown that every graph has a finite outer multiset dimension. It is obvious that a multiset basis of a graph is also its outer multiset basis, and an outer multiset basis of a graph is also its basis. Thus, $dim(G) \leq dim_{ms}(G) \leq md(G)$. The following theorem provides a natural upper bound for the outer multiset dimension of a graph and the necessary and sufficient condition for a graph admitting the said upper bound.\\

\begin{theorem} \cite{GRTY19, KKY23} \label{dmsn-1}
If the order of $G$ is $n$, then $dim_{ms}(G) \leq n-1$. Moreover, $dim_{ms}(G) = n-1$ if and only if $G$ is a regular graph with diameter at most 2.
\end{theorem}

%related to the multiset dimension can be seen in \cite{rino,novi,hafi}. 
\vspace{3mm}
In the case where only neighbors need to be distinguished, Alfarisi \textit{et. al.} \cite{ADKA19} introduced the notion of local multiset dimension. Formally, if for every two adjacent vertices $u$ and $v$ in $V(G)$, $m(u|W)\neq m(v|W)$, then $W$ is called a {\em local multiset resolving set} of $G$. If the graph $G$ has a local multiset resolving set, then this set with the smallest cardinality is called a {\em local multiset basis} and its cardinality is called the {\em local multiset dimension} of $G$, denoted by $lmd(G)$; otherwise, $G$ has an infinite local multiset dimension. Every local multiset basis of a graph induces its local basis, and every multiset basis of a graph induces its local multiset basis. This leads to the relationship $ldim(G) \leq lmd(G) \leq md(G)$.\\

It is then natural to consider the local version of the outer multiset dimension, as in the following. If for every pair of distinct vertices $u,v$ in $V(G)\backslash W$, $m(u|W)\neq m(v|W)$, then $W$ is called a {\em local outer multiset resolving set for $G$}. This set with the smallest cardinality is a {\em local outer multiset basis}, and its cardinality is the {\em local outer multiset dimension of $G$}, denoted by $ldim_{ms}(G)$. Since every outer multiset basis of a graph induces its local outer multiset basis, then $ldim_{ms}(G) \leq dim_{ms}$. Additionally, since every local multiset basis of a graph induces its local outer multiset basis, and every local outer multiset basis of a graph induces its local basis, then $ldim(G) \leq ldim_{ms}(G) \leq lmd(G)$. We summarize all the inequalities involving the dimensions in the following observation.\\

\begin{observation} \label{ineq}
    \begin{enumerate}
    \item $1 \leq ldim(G) \leq dim(G) \leq dim_{ms}(G) \leq md(G)$,
    \item $1 \leq ldim_{ms}(G) \leq dim_{ms}(G) \leq n-1$, and 
    \item $1 \leq ldim(G) \leq ldim_{ms}(G) \leq lmd(G) \leq md(G)$.
    \end{enumerate}
\end{observation}

%However, Ridho et al. \cite{local} have determined the local multiset dimension for some classes of graphs, such as $k-$partites, cycles, completes, and some classes of trees like paths, stars, complete $k-$ary trees, caterpillars.
%The other results can be studied in \cite{uni,homo,sha}  which provide the exact value of the local multiset dimension for some unicyclic graphs, homogeneous pendant graphs, and $m-$shadow graphs, respectively. In a recent work, Ridho et al \cite{note} provides local multiset dimension for hypercubes and cartesian product graphs.

\vspace{3mm}
In this paper, we present some basic properties of the local (outer) multiset dimensions of a graph, which include a sufficient condition for a finite local multiset dimension (Section \ref{finite}) and some sharp bounds for the dimensions (Section \ref{bound}). In the last section (Section \ref{smalldiam}), we consider the local (outer) multiset dimension of graphs with diameter at most two.   
%We also examine the relationship between the local multiset dimension of a graph with respect to its subgraph. In addition, we investigate the local multiset dimension of a subdividing and a joining two graphs. However, we also determine exact values of local multiset dimension for some classes of graphs, such as all trees and non-cycle unicylic graphs, %Petersen graph on 10 vertices, friendships, complete books, and some join of two graphs.

\section{A sufficient condition for a graph with finite local multiset dimension}
\label{finite}

%Denote by ${\mathcal{K}_G}$ the set of all cliques of order at least 3 in $G$. 
For $r\geq 3$, a vertex $u$ in a clique $K_r$ of $G$ is called a {\em $K_r$-end vertex} in $G$ if the degree of $u$ in $G$, $deg_G(u)=r-1$.\\

\begin{theorem}\label{K-end}
If $G$ has a finite local multiset dimension, then every clique $K$ of $G$ contains at most two $K$-end vertices. Moreover, if a clique $K$ of $G$ contains two $K$-end vertices, then exactly one must be in every local resolving set of $G$.
\end{theorem}
\begin{proof}
Suppose that $K=K_r$ is a clique of $G$ containing three distinct $K$-end vertices, say $u,v,$ and $w$. It is clear that $d_G(x,u)=d_G(x,v)=d_G(x,w)$ for any $x\in V(G)\setminus\{u,v,w\}$. Take an arbitrary subset $W\subseteq V(G)$. By the pigeonhole principle, two of $u,v,$ and $w$ are either contained in $W$ or not.  %WLOG, if $u$ and $v$ are contained in $W$, $$m(u|W)=m(u|W\setminus\{u,v\})\cup m(u|\{u,v\})=\{d_G(u,x)|x\in W,x\neq u,v\}\cup \{0,1\}$$ $$=\{d_G(v,x)|x\in W,x\neq v,u\}\cup \{1,0\}=r_m(v|W\setminus\{v,u\})\cup r_m(v|\{v,u\})=r_m(v|W),$$
%(ii) even though if both do not belong to $W$, then $$r_m(u|W)=\{d_G(u,x)|x\in W\}=\{d_G(v,x)|x\in W\}=r_m(v|W).$$ 
In both cases, the two vertices will have the same multiset representation, so $G$ has no local multiset resolving set. 

Let $u$ and $v$ be the only two distinct $K$-end vertices for some $K$ a clique of $G$ and $W$ be an arbitrary set of vertices in $G$. If $u$ and $v$ are both in $W$, then $m(u|W)=m(v|W)$; the same result also occurs when both $u$ and $v$ are not in $W$. Therefore, either $u$ or $v$ must be in $W$. 
\end{proof}

For the local outer multiset resolving set, the following condition involving $K$-end vertices holds.\\ 

\begin{theorem}\label{outK-end}
If a clique $K$ of $G$ contains more than two $K$-end vertices, then all of them, except one, belong to every local outer multiset resolving set of $G$.
\end{theorem}
\begin{proof}
Let $W$ be an arbitrary set of vertices in $G$. Suppose that there exist two distinct $K$-end vertices $u$ and $v$ which are not in $W$. If $w$ is a vertex in $K$ that is not a $K$-end vertex, then $d(u,x) = d(w,x) + 1 = d(v,x)$, for every $x$ in $V(G) \backslash V(K)$. And so $m(u|W)=m(v|W)$.  
\end{proof}

Applying Theorem \ref{K-end} and \ref{outK-end}, we obtain the following local (outer) multiset dimension for vertex and edge amalgamation of complete graphs and corona product of an arbitrary graph and a complete graph. For a sequence of complete graphs $K_{n_1},K_{n_2},\ldots, K_{n_m}$ and a vertex $v_i$ in $K_{n_i}$, the \textit{vertex amalgamation} $Amal(K_{n_i},m)$ is the graph obtained from $K_{n_1},K_{n_2},\ldots, K_{n_m}$ by identifying $v_i$. Similarly, for an edge $e_i$ in $K_{n_i}$, the \textit{edge amalgamation} $EdgeAmal(K_{n_i},m)$ is the graph obtained from $K_{n_1},K_{n_2},\ldots, K_{n_m}$ by identifying the copies of $e_i$.\\ 

\begin{theorem} \label{amal}
Let $m\ge 2$ and $n_i \geq 1$, for every $1\leq i \leq m$.
\begin{enumerate}
    \item The $Amal(K_{n_i},m)$ has finite local multiset dimension if and only if $n_i \leq 3$ for every $1\leq i \leq m$. Moreover, if $m_3$ is the number of $K_3$s, then 
    \[lmd(Amal(K_{n_i},m))=
    \left\{
    \begin{array}{ll}
        1, & \text{if } m_3 = 0,\\
        2, & \text{if } m_3 = 1,\\
        m_3, & \text{if } m_3 \geq 2.
    \end{array}
    \right.\]    
    \item If $m_3$ is the number of $K_3$s and $m_{\geq 4}$ is the number of $K_n$s with $n\geq 4$, then 
    \[ldim_{ms}(Amal(K_{n_i},m))=
    \left\{
    \begin{array}{ll}
        1, & \text{if } m_3 = 0 \ \mathrm{and} \ m_{\geq 4} = 0,\\
        2, & \text{if } m_3 = 1 \ \mathrm{and} \ m_{\geq 4} = 0,\\
        \Sigma_{n_i \geq 3} (n_i-2), & \text{otherwise}.
    \end{array}
    \right.\]    
    %$ldim_{ms}(Amal(K_{n_i},m))=\Sigma_{n_i\geq 3} (n_i-2)$.
\end{enumerate}
\end{theorem}
\begin{proof} For the first part, it is clear that $n_i \leq 3$ for every $1\leq i \leq m$ from Theorem \ref{K-end}. 

If $m_3 = 0$, then $n_i=1$ or $2$ for every $i$ and $Amal(K_{n_i},m)$ is a star, and the singleton set of the center vertex is resolving. If $m_3=1$, it is clear that $Amal(K_{n_i},m)$ is not bipartite and $lmd(Amal(K_{n_i},m))\geq 2$. Consider $W$ a set that contains a leaf and a vertex in $K_3$. It is easy to see that $W$ is resolving and so $lmd(Amal(K_{n_i},m))=2$. If $m_3\geq 2$, then by Theorem \ref{K-end}, $lmd(Amal(K_{n_i},m)) \geq m_3$. Collect a single nonidentifying vertex from each $K_3$ to obtain a local multiset basis for $Amal(K_{n_i},m)$ of cardinality $m_3$.

For the second part, the results for $m_3 = 0$ and $m_{\geq 4} = 0$ or $m_3 = 1$ and $m_{\geq 4} = 0$ follow from the first part of the theorem. For the rest of the case, $ldim_{ms}(EdgeAmal(K_{n_i},m)) \geq \Sigma_{n_i\geq 3} (n_i-2)$ by Theorem \ref{outK-end}. Consider $W$ as the set that contains exactly the $n_i-2$ $K$ end vertices of each $K_{n_i}, n_i\geq 3$. Then, the only pairs of adjacent vertices outside $W$ are the identifying vertex, say $v_0$, and a $K$-end vertex in each $K_{n_i}$. The representation of $v_0$ with respect to $W$ contains only $1$s, while the representation of a $K$-end vertex with respect to $W$ also contains some $2$s. And so, $W$ is a basis.
\end{proof}

\begin{theorem} \label{edgeamal}
Let $m\ge 2$ and $n_i \geq 2$, for every $1 \leq i \leq m$.
\begin{enumerate} 
    \item The $EdgeAmal(K_{n_i},m)$ has finite local multiset dimension if and only if $n_i \leq 4$ for every $1 \leq i \leq m$. Moreover, if $m_2, m_3$, and $m_4$ are the number of $K_2$s, $K_3$s, and $K_4$s, respectively, then 
    \[lmd(EdgeAmal(K_{n_i},m))=
    \left\{
    \begin{array}{ll}
        1, & \text{if } m_2 = m, \\
        3, & \text{if } m_3 \neq 0 \text{ and } m_2 + m_3 = m,\\
        m_4+1, & \text{if } m_4 \neq 0.
    \end{array}
    \right.\]   
    \item If $m_2, m_3$, and $m_{\geq 4}$ are the number of $K_2$s, $K_3$s, and $K_n$s ($n\geq 4$), respectively, then
    \[ldim_{ms}(EdgeAmal(K_{n_i},m))=\left\{
    \begin{array}{ll}
        1, & \text{if } m_2 = m, \\
        2, & \text{if } 1 \leq m_3 \leq 2 \text{ and } m_2 + m_3 = m,\\
        3, & \text{if } m_3 \geq 3 \text{ and } m_2 + m_3 = m,\\
        \Sigma_{n_i\geq 4} (n_i-3) +1, & \text{if } m_{\geq 4} \neq 0.
    \end{array}
    \right.\]
\end{enumerate}
\end{theorem}
\begin{proof}
The necessary and sufficient condition for the finiteness of the local multiset dimension of $EdgeAmal(K_{n_i},m)$ follows directly from Theorem \ref{outK-end}. 

Let $uv$ be the identifying edge in $EdgeAmal(K_2,m)$ for the remaining proof. 

If $m_2=m$, then all the $K_{n_i}$s are $K_2$ and $EdgeAmal(K_2,m)=K_2$. Thus, $lmd(EdgeAmal(K_{n_i},m))=ldim_{ms}(EdgeAmal(K_{n_i},m))=1$.

If $m_3 \neq 0$ and $m_2 + m_3 = m$, then $EdgeAmal(K_{n_i},m)$ is not bipartite; and so $lmd(EdgeAmal(K_{n_i},m))\geq 2$ and $ldim_{ms}(EdgeAmal(K_{n_i},m))\geq 2$. Let $X=\{x,y\}$ be a set of two vertices in $EdgeAmal(K_{n_i},m)$. Consider two cases: (i) $x$ and $y$ are nonadjacent, where $m(u|X)=m(v|X)$, and (ii) $x$ and $y$ are adjacent, 
where $m(x|X)=\{0,1\}=m(y|X)$. Thus, $lmd(EdgeAmal(K_{n_i},m))\geq 3$. Taking a set containing two nonadjacent vertices and $u$ gives us a local multiset basis; and so $lmd(EdgeAmal(K_{n_i},m))=3$. However, if we look for a local outer multiset resolving set, we must consider two subcases of (ii). First, $x=u$ and $y=v$, where all nonadjacent vertices have the same representation multiset $\{1^2\}$, or $x=u$ and $y\neq v$, which, if $m_3=1$ or $2$ will give a local outer resolving set and if $m_3 \geq 3$ will lead to the same representation multiset $\{1,2\}$ for all all nonadjacent vertices.  

If there exists some $i$ with $n_i=4$, then by Theorem \ref{K-end}, $lmd(EdgeAmal(K_{n_i},m)) \geq m_4$. It is clear that one of $u$ and $v$ must be in a local multiset resolving set. Consider $W$ as the set of a single nonidentifying vertex from each $K_4$ together with $u$. Then, $m(u|W)=\{0,1^{m_4}\}$, $m(v|W)=\{1^{m_4+1}\}$, $m(x|W)=\{1^{2},2^{m_4-1}\}$, for $x\notin W \cup \{u,w\}$, and $m(y|W)=\{0,1,2^{m_4-1}\}$, for $y\in W \backslash \{u\}$. Thus, $W$ is local multiset basis for $EdgeAmal(K_{n_i},m)$ of cardinality $m_4 + 1$.

%the only pairs of adjacent vertices are  $v$ and a $K$-end vertex in each $K_{n_i}$. Thus, $m(v|W)$ contains only $1$s, while the representation of a $K$-end vertex with respect to $W$ always contains some $2$s.
%collect to obtain a local multiset basis for $EdgeAmal(K_{n_i},m)$ of cardinality $m_4 + 1$.\\

On the other hand, $ldim_{ms}(EdgeAmal(K_{n_i},m)) \geq \Sigma_{n_i\geq 4} (n_i-3)$ by Theorem \ref{outK-end}. It is clear that one of $u$ and $v$ must be in a local outer multiset resolving set. Consider $W$ as the set that contains exactly the $n_i-3$ $K$-end vertices of each $K_{n_i}, n_i\geq 4$, together with $u$. Then, the only pairs of adjacent vertices outside $W$ are $v$ and a $K$-end vertex in each $K_{n_i}$. Thus, $m(v|W)$ contains only $1$s, while the representation of a $K$-end vertex with respect to $W$ always contains some $2$s. Consequently, $W$ is a local outer multiset basis.
\end{proof}

\vspace{3mm}
The \textit{corona product} of a graph $G$ (of order $n$) and a sequence of graphs $H_i$ ($1\leq i\leq m$), $G \odot H_i$, is the graph obtained by joining the $i$-th vertex of $G$ with an edge to every vertex of $H_i$.\\

\begin{theorem} \label{corona}
 Let $G$ be a graph of order $n \geq 1$ and $m_i \geq 1$, for every $1\leq i \leq m$.
 \begin{enumerate}
    \item If $G \odot K_{m_i}$ has a finite local multiset dimension, then $m_i\leq 2$ for every $1\leq i \leq m$. Moreover, if $m_i\leq 2$ for every $1\leq i \leq m$, then $lmd(G \odot K_{m_i}) \geq m$ and this bound is sharp.  
    \item $ldim_{ms}(G \odot K_{m_i}) \geq \Sigma_{i=1}^m (m_i-1)$ and this bound is sharp.
\end{enumerate}
\end{theorem}
\begin{proof}
Both lower bounds are a direct consequence of Theorems \ref{K-end} and \ref{outK-end}, respectively. The sharpness of the bounds occurs when $G$ is a path of order at least $3$ and $m_i=2$ for every $1\leq i \leq m$.
\end{proof}

\begin{problem}
    For $1\leq i \leq m$ and $m_i \geq 1$, find a "good" upper bound for $lmd(G \odot K_{m_i})$ and $ldim_{ms}(G \odot K_{m_i})$.
\end{problem}

\section{Bounds for the local (outer) multiset dimension of graphs}
\label{bound}

In this section, we investigate the bounds for local (outer) multiset dimensions of graphs. We start by characterizing graphs with the local (outer) multiset dimensions equal to 1, which leads to the fact that 2 is a natural lower bound for the local (outer) multiset dimension of non-bipartite graphs.\\

\begin{theorem}\label{1}
$lmd(G)=ldim_{ms}(G)=1$ if and only if $G$ is bipartite.
\end{theorem}
\begin{proof} 
Since $ldim_{ms}(G) \leq lmd(G)$, the result follows if we can prove that $lmd(G)=1$ if and only if $G$ is bipartite. 

Let $U$ and $V$ be the partite sets of $G$, and $W=\{w\}$, for a vertex $w\in U$. Since $d(u,w)$ is even for all $u\in U$ and $d(v,w)$ is odd for all $v\in V$, then $W$ is a local multiset basis. 

To verify the converse, let $G$ be a non-trivial graph with $lmd(G)=1$, with $W=\{w\}$ a local multiset basis. If $e(w)$ is the eccentricity of $w$, let $N_i$ be the set of vertices of distance $i$ for $w$, for $0 \leq i \leq e(w)$. Since $W$ is a local multiset basis, all $N_i$s are independent sets. Moreover, if $i$ and $j$ are integers with $0\leq i, j\leq e(w)$ and $|i-j|\geq 2$, then no vertex in $N_i$ are adjacent to any vertex in $N_j$. In this case, $G$ is bipartite with partite sets $U=\bigcup_{{\rm even}\ i } N_i$ and $V=\bigcup_{{\rm odd}\ i} N_i$. 
\end{proof}

\begin{corollary} \label{2}
If $G$ is not bipartite, then $lmd(G) \geq 2$ and $ldim_{ms}(G) \geq 2$.     
\end{corollary}

\vspace{3mm}
An example of non-bipartite graphs achieving the lower bounds in Corollary \ref{2} is given in the following theorem.\\ 

\begin{theorem} \label{unicyclic}
If $G$ is a non-cycle unicylic graph, then \[lmd(G) = ldim_{ms}(G) = \left \{\begin{array}{ll}
   1, & {\rm if} \ G \ {\rm contains \ an \ even \ cycle},\\
   2, & {\rm if} \ G \ {\rm contains \ an \ odd \ cycle}. 
\end{array}\right.\]
\end{theorem}
\begin{proof}
If $G$ contains an even cycle, it is bipartite, and the result follows. 

If $G$ contains an odd cycle, then it is not bipartite, and by Corrolary \ref{2}, $lmd(G) \geq 2$ and $ldim_{ms}(G) \geq 2$. Let $C=(c_1,\ldots,c_{2k+1}), k \geq 1$, be the cycle in $G$. If $S=V(G)-V(C)$, then for $1 \leq i \leq 2k+1$, the induced subgraph $T_i=G[S \cup{c_i}]$ is a tree. 

Since $G$ is not a cycle, there exists a $T_i$ that is not trivial, say $T_1$. Denote by $t_1$ a vertex in $G-C$ adjacent to $c_1$, and choose $W=\{t_1,c_2\}$. Thus,
\[m(c_1|W)=\{1,1\}, m(c_2|W)=\{0,2\},  m(c_{k+2}|W)=\{k,k+1\},\] \[m(c_i|W)=\{i-2,i\}, 3\leq i\leq k+1,\
{\rm and} \ m(c_i|W)=\{n-i+2,n-i+2\}, k+3\leq i\leq n.\] 
This means that all adjacent vertices in $C$ have distinct representation multisets.

If $x$ is a vertex in $T_1$, $m(x|W) = \{d_G(x,u_1),d_G(x,u_1)+2\}$; on the other hand, if $x$ is a vertex in $T_i, i\neq 1$, $m(v|W)=m(c_j|W)+d_G(v,c_j)$. Since all $T_i$s are trees, all adjacent vertices in $T_i$ have distinct representation multisets. 

%$r_m(t_1|W)=\{0,2\},$
%It is easily to verify that $r_m(c_i|W)\neq r_m(c_{i+1}|W)$ for each $i$. For any vertex $v$ in $T_j$ (if exists) with $j\neq 1$, we get  Since the distances between every $v_j\in N(v)$ in tree $T_j$ and vertex $c_j$ differ by 1,  $r_m(v|W)\neq r_m(v_j|W)$. Now, for any vertex $x\neq u_1$ in $T_1$,  Similarly as before, also we obtain $r_m(x|W)\neq r_m(x_1|W)$ for every $x_1\in N(x)$. Hence, $W$ is a local resolving set of $G$.
Thus, $W$ is a local multiset resolving set of $G$. Since $ldim_{ms}(G) \leq lmd(G)$, the result follows.
\end{proof}

\vspace{3mm}
Interestingly, when a unicyclic graph is a cycle, its local multiset dimension generally differs from that of a non-cycle unicyclic graph. The following result of the local multiset dimension for cycles has been proved in \cite{ADKA19}. However, we revised the result slightly to incorporate cycles with infinite local multiset dimension, which were missed there.\\  

\begin{theorem}
For $n\geq3$, 
$$lmd(C_n)=\left \{ \begin{array}{ll}
1, {\rm if } \ n \ {\rm is \ even},\\ 
3, {\rm if } \ n \geq 7 \ {\rm is \ odd},\\
\infty, {\rm if} \ n=3,5,
\end{array}
\right.$$
and $$ldim_{ms}(C_n)=\left \{ \begin{array}{ll}
1, {\rm if } \ n \ {\rm is \  even},\\ 
2, {\rm if } \ n \ {\rm is \  odd}.
\end{array}
\right.$$
\end{theorem}
\begin{proof}
Since even cycles are bipartite, Theorem \ref{1} implies that $lmd(C_n)=ldim_{ms}(C_n)=1$ for even $n$. For $n=3$, let $S$ be an arbitrary set with cardinality 2. Thus, the two vertices in $S$ are adjacent, and those two vertices have the same representation multisets with respect to $S$. Here we obtain $lmd(C_3) > 2$. However, since there is only one vertex outside $S$, we have $ldim_{ms}(C_3)=2$. Taking $S=V(C_3)$ results in all vertices have the same representation multisets with respect to $S$, which leads to $lmd(C_3) = \infty$.       
For odd $n\geq 5$, let $S=\{u,v\}$ be an arbitrary set with cardinality 2. Consider two cases: when $u$ and $v$ are adjacent and when they are not adjacent. In the first case, $m(u|S)=m(v|S)$. In the second case, there are two paths connecting $u$ and $v$, one in odd order and the other in even order. We denote the vertices in the even path $u=x_1,x_2, \ldots, x_{2t}=v$. Here, $m(x_t|S)={t-1,t}=m(x_{t+1}|S)$. Therefore, $lmd(C_n) > 2$. Since $md(C_n)=3, n\geq 6$, then $lmd(C_n)=3, n\geq 6$. For $n=5$, taking any subset $S$ of $V(C_5)$ results in two vertices that have the same representation multisets with respect to $S$. This concludes the proof for the first part of the theorem.

For the second part of the theorem, consider $S=\{u,v\}$, where $u$ and $v$ are adjacent. For odd $n \geq 3$, denote by $u=x_1,x_2, \ldots, x_n=v$, the vertices on the longer path connecting $u$ and $v$. Then, $m(x_{\lceil \frac{n}{2} \rceil}|S)= \{ \lfloor \frac{n}{2} \rfloor^2 \}$ and $m(x_i|S)=m(x_{n+1-i}|S)=\{i-1,i\}$. Therefore, no two adjacent vertices have the same representation multisets with respect to $S$, and $ldim_{ms}(C_n)=2$, for odd $n$.
\end{proof}

\vspace{3mm}
The next theorem provides a sharp lower bound for graphs' local (outer) multiset dimensions as a function of the order of the largest complete subgraph.\\
\begin{theorem}\label{clique1}
Let $G$ be a graph and $n\geq 2$. If $\omega(G)$ is the order of a maximal clique in $G$, then
$$ldim_{ms}(G)\geq \lceil \log_2\omega(G)\rceil \ \text{and} \ lmd(G)\geq \lceil \log_2\omega(G)\rceil.$$ 
Moreover, there exists a graph $G$ with $\omega(G)=n$ such that $ldim_{ms}(G)=lmd(G)=\lceil \log_2n\rceil$.
\end{theorem}
\begin{proof}
We shall prove the first inequality, and the second follows from Observation \ref{ineq}.

Let $K_n$ be a maximal clique of $G$. We are done for $n=2$ since $G$ contains $K_2$ and $ldim_{ms}(G)\geq 1$. For $n=3$ and $4$, it follows from Corrolary \ref{2}. 
%\textcolor{red}{Maximal clique di graf bipartite adalah $K_2$, kenapa bisa menggunakan Theorem \ref{1}?Mohon maaf Bu, seharusnya berdasarkan Corollary 1 karena untuk n=3,4 ($\lceil \log_2\omega(G)\rceil=2$) pasti G non-bipartit} 
For $n\geq 5$, suppose that there is a local outer multiset basis $W$ such that $|W|<\lceil\log_2 n\rceil$ or $n>2^{|W|}$. 

%Consider the following two cases.
%\textbf{Case 1.} $W\cap V(K_n)=\emptyset$

Let $W\cap V(K_n)=X$ and $Y=V(K_n)\setminus X$. Obviously, $|Y|=n-|X|\geq 2$. We shall count the number of possible representation multisets of vertices in $Y$ with respect to $W$. For any $y\in Y$, $m(y|W) = m(y|X) \bigcup m(y|W\setminus X) = \{1^{|X|}\} \bigcup m(y|W\setminus X)$. 

Now, consider an arbitrary vertex $w$ in $W\setminus X$. Since there are two possible values of $d_G(y,w)$ for every $y\in Y$, there are at most $2^{|W|-|X|}$ possible representation multisets of vertices in $Y$ with respect to $W\setminus X$, and thus, with respect to $W$. 

Consider two cases. First, when $X=\emptyset$. Since $n>2^{|W|}$ and at most $2^|W|$ possible representation multisets, there are two vertices in $Y$ with the same representation multisets with respect to $W$, a contradiction. The second case is when $X \neq \emptyset$. It can be verified that $2^{|W|}-|X|>2^{|W|-|X|}$ for all $|W|>|X|\geq 1$. Since $n>2^{|W|}$,  $n-|X|>2^{|W|-|X|}$. Consequently, there are two vertices in $Y$ with the same representation multisets with respect to $W$, another contradiction. Therefore, $ldim_{ms}(G)\geq \lceil \log_2\omega(G)\rceil$. %By Observation \ref{ineq}, the desired result follows.

%Take an arbitrary $w$ in $W$. Since all vertices of $K_n$ are adjacent, there are two possible values of $d_G(u,w)$ for every $u\in V(K_n)$: either the shortest path is via $w$ or a neighbor of $w$. Thus, there are at most $2^{|W|}$ possible representation multisets of vertices in $K_n$ with respect to $W$. Since $n>2^{|W|}$, there are two vertices $u$ and $v$ di $K_n$ with the same representation multisets with respect to $W$, a contradiction.

%\textbf{Case 2.} $W\cap V(K_n)\neq \emptyset$

%Let $W\cap V(K_n)=X$ and $Y=V(K_n)\setminus X$. Obviously, $|Y|=n-|X|\geq 2$. \textcolor{red}{Ini kenapa harus $\geq 2$? Kalau bisa $\geq 0$, sebenarnya Case 1 dan 2 bisa digabung, reasoningnya sama.} Take an arbitrary vertex $w$ in $W\setminus X$. 
%Since there are two possible values of $d_G(y,w)$ for every $y\in Y$, there are at most $2^{|W|-|X|}$ possible representation multisets of vertices in $Y$ with respect to $W\setminus X$. Moreover, for every $y\in Y$ $m(y|X)=\{1^{|X|}\}$, and so $m(y|W)=m(y|X)\cup m(y|W\setminus X)$. Thus, there are at most $2^{|W|-|X|}$ possible representation multisets of vertices in $Y$ with respect to $W$. 
%Algebraically, it can be verified that $2^{|W|}-|X|>2^{|W|-|X|}$ for all $|W|>|X|\geq 1$ \textcolor{red}{Apa ini bisa diganti $\geq 0$ supaya jadi 1 kasus?}. 
%Using this fact, $n>2^{|W|}$ implies $n-|X|>2^{|W|-|X|}$. Consequently, there are two vertices $u$ and $v$ in $Y$ with the same representation multisets with respect to $W$, which leads to a contradiction.

Finally, to prove that the lower bounds are sharp, for $n=2$, choose $G=K_2$. For $n\geq 3$, we consider two cases: %(i) $n=2^k$, for some $k$ and (ii) $2^{k-1} < n < 2^k$, for some $k$. 

\textbf{Case 1 ($n=2^k$, for some $k$).} We construct a graph $G$ containing a clique $K_{2^k}$ with vertices $u_1, \ldots, u_k, v_0, v_1, \ldots, v_{2^k-k-1}$. For $1 \leq i \leq 2^k-k-2$, label the vertex $v_i$ with a $k$-vector $(a_{i,1}, a_{i,2}, \ldots, a_{i,k})$, as follows:
\[a_{i,j} = \left \{\begin{array}{cc}
   2j & {\rm if} \ i = j,\\
   2j+1 & {\rm if} \ i \neq j. 
\end{array}\right.\]
While $v_1$ is labeled with $(2, 4, \ldots, 2k)$ and $v_{2^k-k-1}$ with $(3, 5, \ldots, 2k+1)$.
To complete the construction of $G$, for $1 \leq j \leq k$, attach to each $u_j$ a path $P_j$ of length $2j$, and denote the vertices in $P_j$ by $u_j, u_{j,1}, u_{j,2}, \ldots, u_{j,2j}$. Lastly, add an edge between $u_{j,1}$ and $v_i$ whenever $a_{i,j}=2j$, for $1 \leq i \leq 2^k-k-1, 1 \leq j \leq k$. 

Choose $W=\{u_{1,2}, u_{2,4}, \ldots, u_{2,2k}\}$ and we will show that $W$ is a local multiset resolving set for $G$. 
Since $v_i$ is adjacent to $u_{j,1}$  whenever $a_{i,j}=2j$, then $m(v_i|W)=\{a_{i,1}, a_{i,2}, \ldots, a_{i,k}\}$, and so all $v_i$s have distinct representation multisets. On the other hand, for $1 \leq l,j \leq k$ the distance from $u_j$ to $u_{l,2l}$ is $2l$, if $j=l$, and $2l+1$, otherwise. This leads to the fact that all $u_j$s have distinct representation multisets. Since the entries of $m(v_0|W)$ are all even, the entries of $m(v_{2^k-k-1}|W)$ are all odd, the entries of $m(v_i|W), 1 \leq i \leq 2^k-k-2$ are all even except for one, and the entries of $m(u_i|W), 1 \leq i \leq k$, are all odd except for one, then all the vertices in the clique have distinct representation multisets.

Now, consider the vertices of the paths $P_j$ s. For $1 \leq j \leq k$, the distance from $u_{j,1}$ to $u_{l,2l}$ is $2l-1$, if $j=l$, and $2l+1$, otherwise. Thus, $m(u_{j,1}|W)$ contains all odd entries, while its $k+1$ neighbours in the clique always have an even entry in their representation multisets. As to the other pairs of adjacent vertices on the paths $ P_j$, it is obvious that they have distinct representation multisets.

Thus, $W$ is a local multiset resolving set and a local outer multiset resolving set for $G$. Combining with the result in the first part of this theorem, we obtain $lmd(G) = ldim_{ms}(G) = k = \big\lceil \log_2n\big\rceil$. 

An example of the construction for $k=3$ can be viewed in Figure \ref{gbclique}.

\begin{figure}[phtb]
\begin{center}
    \includegraphics[height=1.4in]{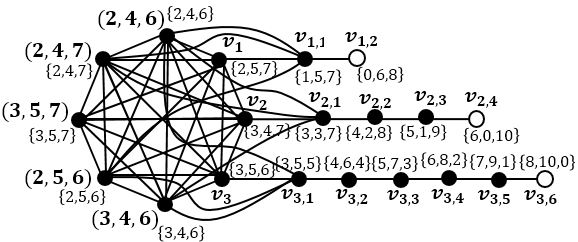}
    \caption{The graph $G$ with $\omega(G)=8$ and $lmd(G)=ldim_{ms}=3$.} \label{gbclique} 
\end{center}
\end{figure}

\textbf{Case 2 ($2^{k-1}<n<2^k$).} We construct a graph $G'$ from $G$ in Case 1 by deleting the vertices $v_1, \ldots, v_{2^k - n}$. Here, the set $W$ is also a local multiset resolving set and a local outer multiset resolving set for $G'$.
\end{proof}

%\begin{theorem} \label{clique} 
%Let $G$ be a graph. If $\omega(G)$ is the order of a maximal clique in $G$, then $$lmd(G)\geq \lceil \log_2\omega(G)\rceil.$$ 
%Moreover, for $n\geq 2$, there exists a graph $G$ with $\omega(G)=n$ such that $lmd(G)=\big\lceil \log_2n\big\rceil$. 
%\end{theorem}
%\textcolor{red}{Cek apakah berlaku untuk local outer multiset.? Ternyata setelah saya periksa lagi juga berlaku dengan bukti yang lebih rumit. Saya tulis di Teorema \ref{clique1}}
%\begin{proof}
%Let $K_r$ be a maximal clique in $G$ and  $W=\{w_1,\ldots,w_k\}$ be a local multiset basis of $G$. For an arbitrary $i, 1\leq i \leq k$, let $u$ and $v$ be two vertices in $K_r$ such that $d_G(v,w_i)\leq d_G(u,w_i)$. Since $u$ and $v$ are adjacent, $d_G(u,w_i)=d_G(v,w_i)$ or $d_G(u,w_i)=d_G(v,w_i)+1$. Thus, for every vertex $u$ in $K_r$, there are at most two possible values of $d_G(u,w_i)$; and so $|V(K_r)| \leq 2^k$ and $\mu_l(G)=k\geq \log_2r=\log_2\omega(G)$. 
%\end{proof}

\vspace{3mm}
The {\em chromatic number} $\chi(G)$ of a graph $G$ is the minimal number of colors needed to color the vertices in such a way that no two adjacent vertices have the same color.  The {\em diameter} of a graph $G$ is the maximum distance among all the distances of two vertices in $G$. The following two results show that lower bounds of the local (outer) multiset dimension do, in fact, depend on these two parameters.\\ 

\begin{theorem}\label{chromatic}
Let $G$ be a graph with diameter $d\geq 2$. If $g(d,\chi(G))$ is the smallest number $k$ such that $${k+d-1\choose d-1}+{k+d-2\choose d-1}-d+1\geq \chi(G),$$
then $lmd(G)\geq g(d,\chi(G))$ and $ldim_{ms}(G)\geq g(d,\chi(G))$. Moreover, these bounds are sharp.
\end{theorem}
\begin{proof}
Let $W=\{w_1,\ldots,w_k\}$ be a local (outer) multiset basis for $G$. Consider all possible representation multisets of vertices in $G$. For $v\not\in W$, $m(v|W)=\{1^{m_1},\ldots,d^{m_d}\}$, where $m_1+\ldots+m_d=k$,  which leads to at most ${k+d-1 \choose d-1}$ possibilities. For $w_i\in W$, $m(w_i|W)=\{0,1^{n_1},\ldots,d^{n_{d}}\}$, where $n_1+\ldots+n_{d}=k-1$, and so there are at most ${k-1+d-1 \choose d-1}$ possibilities. However, it is impossible that all of $\{0,1^{k-1}\}$, $\{0,2^{k-1}\}$, $\ldots$, $\{0,d^{k-1}\}$ appear simultaneously (only at most one is allowed), and so, there are at most ${k+d-1\choose d-1}+{k+d-2\choose d-1}-d+1$ possible representation multisets of vertices in $G$. Since every pair of adjacent vertices must have distinct representation multisets, ${k+d-1\choose d-1}+{k+d-2\choose d-1}-d+1 \geq \chi(G)$. 

These bounds are sharp for all bipartite graphs (with chromatic number 2), as shown in Theorem \ref{1}.
\end{proof} 

\vspace{3mm}
The sharpness of the bound in Theorem \ref{chromatic} is only known for graphs with chromatic number 2, and thus the following question.\\

\begin{problem}
For any integer $c \geq 3$, is there a graph $G$ with chromatic number $\chi(G)=c$ such that $lmd(G)=g(d,c)$? $ldim_{ms}(G)=g(d,c)$?
\end{problem}

\vspace{3mm}
We obtain the following by applying Theorem \ref{chromatic} to graphs with a small diameter.\\

\begin{corollary}\label{diam2,3}
\begin{enumerate}
    \item If $G$ has diameter 2, then $lmd(G) \geq \frac{\chi(G)}{2}$ and $ldim_{ms}(G) \geq \frac{\chi(G)}{2}$.
    \item If $G$ has diameter 3, then $lmd(G)\geq\sqrt{\chi(G)+2}-1$ and $ldim_{ms}(G) \geq\sqrt{\chi(G)+2}-1$.
\end{enumerate}
\end{corollary}

\vspace{3mm}
We conclude this section by presenting an upper bound for the local (outer) multiset dimension of graphs. Here, a subgraph $H\subseteq G$ is called a {\em maximal subgraph with property $P$} if $H+e$ is not $P$ for any edge $e$ in $G$.\\ 

% cycle. As a rather general result, we then recognize that there is a convenient relationship among the local multiset dimensions of a graph and its maximal subgraph, especially those whose several vertices of degree two. Normally, a vertex of degree two is called {\em a leaf}. For a graph $G$,

\begin{theorem}\label{maxsubgraph}
Let $H$ be a maximal subgraph of $G$ without a leaf. 
\begin{enumerate}
    \item If the local multiset dimension of $H$ is finite, then $lmd(G) \leq lmd(H)$, and
    \item $ldim_{ms}(G) \leq ldim_{ms}(H)$.
\end{enumerate}
Moreover, these bounds are sharp.
\end{theorem}
\begin{proof}
If $G$ has no leaves, then $H=G$, and we are done since $lmd(G) = lmd(H)$ and $ldim_{ms}(G) = ldim_{ms}(H)$. Thus, we assume that $G$ has some leaves.

{\bf Claim 1.} All components in $G-H$ are trees.

Assume that there is a nontree component $F$ of $G-H$, which means that $F$ contains a cycle $C$. Let $\mathcal{P}$ be the set of all shortest paths, where each of them has one end in $V(H)$ and the other in $V(C)$. Since $G$ is connected, $\mathcal{P}$ is not empty. As a consequence, $H\cup {\mathcal{P}}\cup C$ is a subgraph of $G$ without a leaf, a contradiction to the maximality of $H$. \qed 
%It should be that all components $G-H$ are trees.

For $A,B\subseteq V(G)$, we denote by $E(A,B)=\{ab\in E(G)|a\in A,b\in B\}$. As $G$ is connected, this set is not empty.

{\bf Claim 2. For every tree component $T$ of $G-H$, $|E(V(H),V(T))|=1$.}

Suppose that there exists a component $T$ of $G-H$, such that $|E(V(H),V(T))| \geq 2$. Consider the following two cases:

\textit{Case 1.} There are two distinct vertices $t_1,t_2$ in $T$ having neighbours in $H$. By the connectivity of $T$, there exists a unique path connecting $t_1$ and $t_2$, say $_{t_1}P_{t_2}$. This implies that the subgraph $H \bigcup E(V(H),\{t_1,t_2\}) \bigcup$ 
 $_{t_1}P_{t_2}$ has no leaf, that contradicts the maximality of $H$.

\textit{Case 2.} There is exactly one vertex $t$ in $T$ adjacent to at least two distinct vertices in $H$. Again, we obtain a contradiction due to $H \cup E(V(H),\{t\})$ is a subgraph without a leaf. \qed

Now, let $T_1, T_2, \ldots, T_k$ be all the tree components of $G-H$, in which $t_i$ is the vertex in $T_i$ adjacent to the vertex $h_i$ in $H$. Since the local multiset dimension of $H$ is finite, let $W$ be a local multiset basis of $H$. This means that $m(u|W) \neq m(v|W)$ for every pair of adjacent vertices $u$ and $v$ in $H$. 
Let $x,y$ be two adjacent vertices in $T_i$. Since $T_i$ is a tree, one of them, say $x$, should be on the path connecting $t_i$ to the other vertex, here $y$. Thus, $d_{T_i}(y,t_i) = d_{T_i}(x,t_i)+1$. Since $m(t_i|W) = m(h_i|W)+1, 1 \leq i \leq k$, then   
%; if not, it should be ${t_i}Px\cup xy\cup\  yP{t_i}\subseteq T_i$ as a cycle which contradicts with claim (i). 
%Therefore, 
$m(x|W) \neq m(y|W)$, and $W$ is also a local resolving set for $G$.

A similar argument holds when we consider $W$ a local outer multiset basis of $H$.

The sharpness of the bounds follows if we consider $G \equiv H \odot K_2$. It is clear that $H$ is the maximal subgraph of $G$ without a leaf. For every leaf $y$ in $G$ and every vertex $x$ in $H$, $d_G(y,x)=d_H(y',x)+1$, where $y'$ is the neighbor of $y$ in $G$. Thus, the local (outer) multiset basis of $H$ is also the local (outer) multiset basis of $G$.
\end{proof}

\section{Local (outer) multiset dimension of graphs with small diameter}
\label{smalldiam}

In this last section, we consider graphs of diameter at most two. By Theorem \ref{infmd}, these graphs have infinite multiset dimension. If the graphs are regular, they have an outer multiset dimension one less than their order (Theorem \ref{dmsn-1}).\\ 

For graphs of diameter one, the complete graphs, it has been proved that $lmd(K_2)=1$ and $lmd(K_n)=\infty, n \geq 3$ \cite{ADKA19}. By Theorem \ref{dmsn-1}, for $n\geq 2, ldim_{ms}(K_n)=dim_{ms}(K_n)=n-1$.\\

We have determined the local (outer) multiset dimension of two families of graphs with diameter two in Corollaries \ref{amal} and \ref{edgeamal}, and we can see that the majority of graphs in the two families still have infinite local multiset dimension. In the following, we study a family of graphs with diameter two; half of which still admit infinite local multiset dimension.\\  

Let $W_n$ be the wheel on $n+1$ vertices. We denote by $v_p$ the center of $W_n$ 
%Perhatikan bahwa $d(v_p,u)=1$ untuk setiap $u\in W_m\setminus\{v_p\}.$ Oleh karena itu untuk meminimumkan kardinalitas dari himpunan pembeda lokal $W$ dari $W_m,$ $v_p$ tidak pernah dipilih menjadi anggota dari $W$.
and by $W'_n = W_n[V(W_n) \setminus \{v_p\}]$ the induced cycle on $n$ vertices in $W_n$. The local metric dimension of the wheels was determined in \cite{FC21} as follows, and, by Observation \ref{ineq}, they serve as lower bounds for the local multiset dimension of the wheels.\\

\begin{theorem} \cite{FC21} \label{ldimWn}
The local metric dimension of wheels are $ldim(W_3)=3$, $ldim(W_4)=2$, and $ldim(W_n)=\lceil\frac{n}{4}\rceil$, for $n\geq 5$.    
\end{theorem}

\vspace{3mm}
The following lemma is needed to construct a local multiset basis of a wheel.\\

\begin{lemma} \label{1or3}
Let $W$ be a local multiset resolving set of $W_n$. Then
\begin{enumerate}
    \item If $W'_n[W]$ contains a path $P_m$, then either $m=1$ or $m=3$.
    \item If $W'_n[W_n \setminus W]$ contains a path $P_m$, then either $m=1$ or $m=3$.
\end{enumerate}
\end{lemma}
\begin{proof} Let $|W|=k$.
\begin{enumerate}
\item Suppose that $W'_n[W]$ contains a path $P_2$, then the two adjacent vertices in $P_2$ have the same representation multisets: either $\{0, 1, 2^{k-2}\}$ (if $v_p \notin W$)  or $\{0, 1^2, 2^{k-3}\}$ (if $v_p \in W$). Now, suppose that $W'_n[W]$ contains a path $P_m$ with $m\geq 4$. Consider a pair of adjacent vertices that are not end-vertices in $P_m$, say $x$ and $y$. Then $x$ and $y$ have the same representation multisets: either $\{0,1^2,2^{t-3}\}$ (if $v_p \notin W$) or $\{0,1^3,2^{t-4}\}$ (if $v_p \in W$). 
\item Suppose that $W'_n[W_n \setminus W]$ contains a path $P_2$, then the two adjacent vertices in $P_2$ have the same representation multisets: either $\{1, 2^{k-1}\}$ (if $v_p \notin W$)  or $\{1^2, 2^{k-2}\}$ (if $v_p \in W$). Now, suppose that $W'_n[W_n \setminus W]$ contains a path $P_m$ with $m\geq 4$. Consider a pair of adjacent vertices that are not end-vertices in $P_m$, say $x$ and $y$. Then $x$ and $y$ have the same representation multisets: either $\{2^{k}\}$ (if $v_p \notin W$) or $\{1,2^{k-1}\}$ (if $v_p \in W$). 
\end{enumerate}
\end{proof}

\vspace{3mm}
By reasoning similar to the proof of Lemma \ref{1or3}, we obtain the following property for a local outer multiset resolving set of a wheel.\\  

\begin{lemma} \label{out1or3}
Let $W$ be a local outer multiset resolving set of $W_n$. If $W'_n[W]$ contains a path $P_m$, then either $m=1$ or $m=3$.
\end{lemma}

\vspace{3mm}
Now, we are ready to determine the local (outer) multiset dimensions of wheels.\\

\begin{theorem} \label{wheel} For $n \geq 3$,
$$lmd(W_n)=\left \{ \begin{array}{cc}
3, \ {\rm if}\ n=4,6,\\
\lceil\frac{n}{4}\rceil, \ {\rm if}\ n\geq 8 \ {\rm even},\\
\infty, \ {\rm otherwise,}
\end{array}
\right.$$
and
$$ldim_{ms}(W_n)=\left \{ \begin{array}{cc}
3, \ {\rm if}\ n=3,4,6,\\
\lceil\frac{n}{4}\rceil, \ {\rm if}\ n\geq 8 \ {\rm even \ or} \ n \equiv 1 \mod 4,\\
\lceil\frac{n}{4}\rceil + 1, \ {\rm otherwise}.
\end{array}
\right.$$
\end{theorem}
\begin{proof}
For the first part of the theorem, let $W$ be a local multiset resolving set of $W_n$. By Lemma \ref{1or3}, a path in $W'_n[W]$ and a path in $W'_n[W_n \setminus W]$ must come in pairs. Since each is of odd order, the pair of these paths is of even order. And so, every $W_n$ with odd $n$ has no local multiset resolving set. 

For $n=4$, by Observation \ref{ineq} and Theorem \ref{ldimWn}, $lmd(W_4) \geq 2$. Suppose that $W=\{u,v\}$. If $d(u,v)=1$, $m(u|W)=m(v|W)=\{0,1\}$, and if $d(u,v)=2$, there exists a vertex $w$ with $d(u,w)=d(v,w)=1$ such that $m(w|W)=m(v_p|W)=\{1^2\}$. Now, choose $W=\{u,v,w\} \subset W'_4$. It is easy to see that $W$ is a local multiset resolving set for $W_4$.

For $n=6$, by Observation \ref{ineq} and Theorem \ref{ldimWn}, $lmd(W_6) \geq 2$. Suppose that $W=\{u,v\}$. If $d(u,v)=1$, $m(u|W)=m(v|W)=\{0,1\}$, if $d(u,v)=2$, there exists a vertex $w$ with $d(u,w)=d(v,w)=1$ such that $m(w|W)=m(v_p|W)=\{1^2\}$, and $d(u,v)=3$ is impossible due to Lemma \ref{1or3}. Now, choose $W=\{u,v,w\}$ with $d(u,v)=d(u,w)=d(v,w)=2$. It is easy to see that $W$ is a local multiset resolving set for $W_6$.

For even $n \geq 8$, by Observation \ref{ineq} and Theorem \ref{ldimWn}, $lmd(W_n) \geq \lceil\frac{n}{4}\rceil$. Consider $P_{n_1}, P_{m_1}, P_{n_2}, P_{m_2}, \ldots, P_{n_k}, P_{m_k}$ are paths in $W'_n$, where $P_{n_i} \subseteq W'_n[W]$ and $P_{m_i}\subseteq W'_n[W_n \setminus W], 1 \leq i \leq k$. By Lemma \ref{1or3}, to minimize the cardinality of $W$, we should set $n_i=1$ dan $m_i=3, 1 \leq i \leq k$ (for $n \equiv 0 \mod 4$) and $n_i=1, 1 \leq i \leq k, m_i=3, 1 \leq i \leq k-1$, and $m_k=1$ (for $n \equiv 2 \mod 4$). Thus, $lmd(W_m) = \lceil\frac{m}{4}\rceil$.

For the second part of the theorem, when $n=3$, $W_3 \cong K_4$, and so, by Theorem \ref{dmsn-1}, $ldim_{ms}(W_3)=ldim_{ms}=(K_4)=3$. For $n \geq 4$ even, the local outer multiset dimension of $W_n$ follows from Observation \ref{ineq}, Theorem \ref{ldimWn}, and its local multiset dimension from the first part of this theorem.

For $n \geq 5$ odd, %by Observation \ref{ineq} and Theorem \ref{ldimWn}, $ldim_{ms}(W_n) \geq \lceil\frac{n}{4}\rceil$. C
consider $P_{n_1}, P_{m_1}, P_{n_2}, P_{m_2}, \ldots, P_{n_k}, P_{m_k}$ are paths in $W'_n$, where $P_{n_i} \subseteq W'_n[W]$ and $P_{m_i}\subseteq W'_n[W_n \setminus W], 1 \leq i \leq k$. 

For $n \equiv 1 \mod 4$, by Lemma \ref{out1or3}, to minimize the cardinality of $W$, we should set $n_i=1$ dan $m_i=3, 1 \leq i \leq k$. This process covers $n-1 \equiv 0 \mod 4$ vertices, and the last vertex should be located in $W$; otherwise, there exists a path $P\subseteq W'_n[W_n \setminus W]$ of order $4$, which is impossible by Lemma \ref{out1or3}. Therefore, $lmd(W_m) = \lceil\frac{m}{4}\rceil + 1$. 

For $n \equiv 3 \mod 4$, by Lemma \ref{out1or3}, to minimize the cardinality of $W$, we should set $n_i=1, 1 \leq i \leq k, m_i=3, 1 \leq i \leq k-1$, and $m_k=1$. Again, this process covers $n-1 \equiv 2 \mod 4$ vertices, and the last vertex should be located in $W$; otherwise, there exists a path $P\subseteq W'_n[W_n \setminus W]$ of order $2$, which is impossible according to Lemma \ref{out1or3}. Therefore, $lmd(W_m) = \lceil\frac{m}{4}\rceil + 1$.
\end{proof}

\vspace{3mm}
The problem of determining the local (outer) multiset dimensions of graphs with diameter 2 is still widely open, especially since it is well known that almost all graphs have diameter 2. One possible direction to partially solve the problem is by considering the graph joins, which have diameters of at most 2.\\

\begin{problem}
Let $G$ and $H$ be two graphs.
\begin{enumerate}
    \item Determine $lmd(G+H)$ as a function of $lmd(G)$ and $lmd(H)$.
    \item Determine $ldim_{ms}(G+H)$ as a function of $ldim_{ms}(G)$ and $ldim_{ms}(H)$.
\end{enumerate}
    
\end{problem}

\section*{Acknowledgements}

This research is supported by Penelitian Kerjasama Dalam Negeri 2023, funded by the Indonesian Ministry of Education, Culture, Research, and Technology.
% Please use the AMS Method for Preparing References
% to a simple guide one may see http://www.atmos.washington.edu/EDC/Refstyl.pdf

\bibliography{sn-bibliography}% common bib file
%\input output.bbl

\end{document}